\documentclass[11pt,fleqn,leqno]{amsart}

\usepackage{a4wide}
\usepackage[utf8]{inputenc}
\usepackage[T1]{fontenc}
\usepackage[english,activeacute]{babel}
\usepackage{hyphenat}

\usepackage{mathtools}

\usepackage{amsmath}
\usepackage{amssymb}
\usepackage{amsthm}
\usepackage{amsfonts}
\usepackage{amsmath}
\usepackage{latexsym}
\usepackage{euscript}
\usepackage{graphicx}
\usepackage[table]{xcolor}
\usepackage{listings}
\usepackage{url}
\usepackage{enumerate}
\usepackage{blindtext}
\usepackage{hyperref}
\usepackage{xcolor}
\usepackage{multicol}
\usepackage{tikz}
\usetikzlibrary{arrows,automata,positioning}

\usepackage{url}
\usepackage[shortlabels]{enumitem}
\setenumerate{itemsep=0.005ex}
\setitemize{itemsep=0.005ex}
\usepackage{euscript}
\usepackage{hyperref}
\usepackage{url}
\usepackage{comment}
\usepackage{xcolor}
\usepackage{verbatim}
\usepackage{tikz}
\usepackage{graphicx}

\newtheorem{theorem}{Theorem}
\newtheorem*{theorem*}{Theorem}

\newtheorem*{proposition*}{Proposition}
\newtheorem{lemma}{Lemma}

\theoremstyle{definition}

\newtheorem*{definition*}{Definition}

\newtheorem*{remark*}{Remark}
\newtheorem*{observation*}{Observation}

\newenvironment{example*}{\noindent{\bf Example.}}{$\Box$\par}

\newcommand{\N}{\mathbb N}

\newcommand{\Nu}{occ}
\newcommand{\U}{occ}
\renewcommand{\iff}{\text{ if and only if }}
 \everymath{\displaystyle}
\allowdisplaybreaks

\begin{document}

\author{Verónica Becher \and Nicole Graus}

\title{The discrepancy of the Champernowne constant}

\date{\today}



\begin{abstract}
A number is normal in base $b$
if, in its base 
$b$ expansion, all blocks of digits of equal length have the same asymptotic frequency. 
The rate at which a number approaches normality is quantified by the classical 
notion of discrepancy, which indicates how far the scaling of the number by powers of 
$b$ is from being equidistributed modulo 1. 
This rate is known as the discrepancy of a normal number.
The Champernowne constant 
$c_{10} =0.12345678910111213141516\ldots$ 
 is the most well-known example of a  normal number.
In 1986, Schiffer provided the discrepancy of numbers in a family that includes 
the Champernowne constant. His proof relies on exponential sums. 
Here, we present a discrete and elementary proof specifically for 
the discrepancy of the Champernowne constant.
\end{abstract}

\maketitle

\tableofcontents

\section{Normal numbers}

More than a hundred years ago, Émile Borel defined the property of normality of real numbers:
a real number is normal in a given integer base  $b$ if in its  expansion in base $b$ 
all digits have  the same asymptotic frequency and furthermore, 
all blocks of digits of equal length  have the same asymptotic frequency.
Borel proved that almost all real numbers, with respect to  Lebesgue's measure,  are normal in all integer bases
greater than or equal to~$2$.
A nice version of this proof appears in Hardy and Wright's book~\cite[Theorem 148]{HW}.

Borel would have liked to give an example of a normal number that
 is one of the mathematical constants such as  $\pi$,
 or $e$, or $\sqrt{2}$. But so far none of these 
has been proved normal in any base. It remains an open problem \cite{Bugeaud,BC}.

The best known example of a normal number  is Champernowne constant,
\[
c_{10}= 0.123456789101112131415161718192021222324252627... 
\]
Its expansion is  the concatenation of 
all positive integers expressed in base $10$ in increasing order. 
David Champernonwne defined it  specifically to be an example of a number that is normal in base $10$ \cite{champernowne}. He did this work in 1933 with the supervision of G.H. Hardy while he was a student at King's College, Cambridge  \cite{H}.
Champernowne's proof   is elementary, based on a rigorous counting. 

Notice that $c_{10}$  expressed in base $b$, for  $b\not=10$, 
is different from  the Champernowne constant~$c_b$  which is the concatenation of the positive integers expressed in base $b$, in increasing order.
Figure~\ref{fig:c} depicts the  expansion Champernowne number constant $c_{10}$ expressed in base $10$, in base $2$ and in base~$6$.  
It  is not known whether $c_{10}$ is normal to any integer base other than $10$.  

\begin{figure}
\begin{tabular}{ccc}
\includegraphics[width=0.3\textwidth]{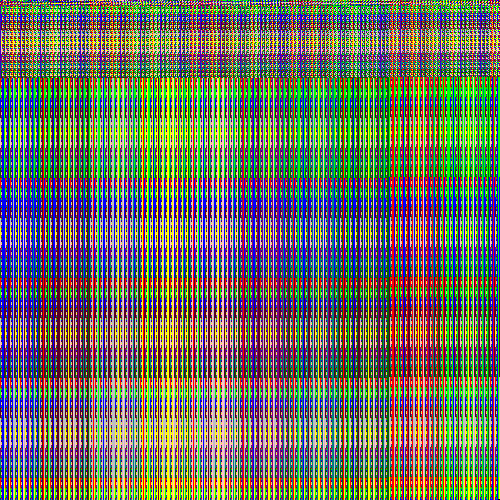}
&
\includegraphics[width=0.3\textwidth]{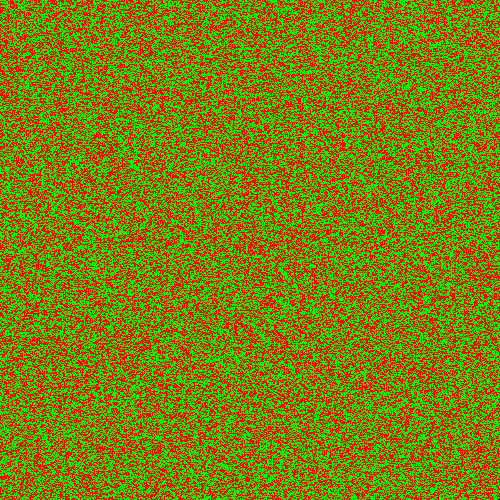}
&
\includegraphics[width=0.3\textwidth]{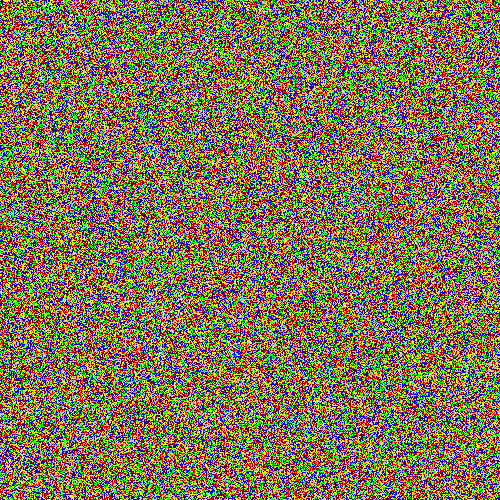}
\end{tabular}
\caption{Plot of the expansion of the first 250000 digits of Champernowne constant in base $10$, in base $2$ and in base $6$, from left to right. In each  base each digit is assigned a different color, and the digits in the expansion  are displayed in row-major order.}
\label{fig:c}
\end{figure}

In this note we present all the definitions and results in base $10$.
It is is equally possible to make the presentation for an arbitrary base $b$
greater than or equal to $2$.

For any real number  $\alpha \in [0,1)$, its expansion in base $10$ is the sequence of digits
$(\alpha_i)_{i\in\N}$ such that $\alpha_i\in \{0,..,9\}$ and 
\[
\alpha=\lfloor \alpha\rfloor+\sum_{i\geq 1} \alpha_i 10^{-i}.
\]
In case $\alpha$ is a rational number, it may have two expansions and we choose the one ending with a tail of $0$s.

\begin{definition*}[Ocurrences counter]
Let $k \in \N$ and let $B=(b_1 \ldots b_k)$ be a block of digits~$b_i \in \{0, \ldots, 9\}$ of length $k$.
    Let $N \in \N$. We define 
    $\Nu(\alpha,B, N)$ as the number of occurrences of the block $B$ 
in $(\alpha_1 \: \alpha_2 \ldots \alpha_N)$ as $k$ consecutive digits,
\[
\Nu(\alpha,B,N)=\#\{i: \alpha_i\ldots \alpha_{i+k-1}=b_1\ldots b_k, 1\leq i\leq N-k+1\}.
\]
\end{definition*}

\begin{example*}
  For  $N=20$, $B=(1\:3\:1)$ and 
  $
   \alpha = 0. 13311321 \mathbf{13131}75 \mathbf{131} 9 \! \underset{\substack{\uparrow \\ \alpha_N}}{1} \! 31
  $
 $\Nu(\alpha,B, N)=3$.\qedhere
\end{example*}


\begin{definition*}[Normal number in base $10$]
    A real number  $\alpha$ is \textit{normal in base 10} if for all $k \in \N$ 
and for every $B$ block of digits of length $k$,
    \[
    \lim_{N \to \infty} \frac{\Nu(\alpha,B, N)}{N} = \frac{1}{10^k}.
    \]
\end{definition*}

\begin{example*}
The rational number 
$0.0123456789 \:  0123456789 \: 0123456789 \ldots
$
is not normal because although the frequency of each digit is $1/10$, the frequency of the block $(1\:1)$ is~$0$.\qedhere
\end{example*}

\section{Discrepancy estimate of normal numbers}

Normality can be expressed
 in the theory of  uniform distribution of sequences modulo~$1$.
A number is normal in a given base $b$ if the scaling of the number by powers of $b$ is uniformly distributed modulo $1$, \cite{Wall}.
 The rate at which a number approaches normality in a base $b$ is given by  the classical notion of  discrepancy.
 
The discrepancy of a sequence $(x_n)_{n\in\N} \subset [0,1)$  measures how far is the sequence from being equidistributed in the unit interval.

\begin{definition*}[Discrepancy of a sequence]
For 
a sequence $( x_{n})_{n \in \N}\subset [0,1)$
the \textit{discrepancy} of it  first $N$ terms~is
    \[
    D \left( (x_n)_{n \in \N} \:, N \right) = \sup_{0\leq a<b<1} 
    \left| \frac{\# \{ n : x_n\in[a,b), 1\leq n\leq N\}}{N} - (b-a) \right|.
    \]
\end{definition*}

For normality in base $10$ we consider for each real number, the scaling of the number by increasing powers of $10$.
If the expansion of $\alpha$ in base~$10$ is given by $\alpha = 0.\alpha_1 \: \alpha_2 \: \alpha_3 \ldots$, we consider the sequence $(x_n)_{n \in \N}$  where
 \begin{alignat*}{3}
 x_1 \quad &= \quad 10^0 \alpha \mod 1 &&= \quad 0.\alpha_1 \: \alpha_2 \: \alpha_3 \ldots \\
 x_2 \quad &= \quad 10^1 \alpha \mod 1 &&= \quad 0.\alpha_2 \: \alpha_3 \: \alpha_4 \ldots \\
 &\:\:\vdots \\
 x_n \quad &= \quad 10^{n-1} \alpha \mod 1 \quad &&= \quad 0.\alpha_n \: \alpha_{n+1} \: \alpha_{n+2} \ldots 
\end{alignat*}

\begin{definition*}[Discrepancy of a number for base $10$]
\mbox{For $\alpha\in[0,1)$,  its discrepancy for base $10$ is}
\[
 D(\alpha, N) = D\left((10^{n-1} \alpha \mod 1)_{n \in \N} \:, N\right).
\]
\end{definition*}
\begin{proposition*} 
A real number $\alpha$ is normal in base $10$ if and only if $\lim_{N \to \infty} D(\alpha, N) = 0$.
\end{proposition*}

\begin{example*}
    Let $\alpha \in [0,1)$ and let $(\alpha_i)_{i\in\N}$ the sequence of digits in its decimal expansion.
   Partition the unit interval $[0,1)$ into  intervals of size $1/10$.
To  ask in which of those ten intervals  is $\alpha$ 
is equivalent to determine  $\alpha_1$,
\[  
\begin{array}{lcl}
\alpha \in  \left[0. \frac{1}{10} \right) &\iff& \alpha_1 = 0 
        \\
\alpha \in  \left[ \frac{1}{10}, \frac{2}{10} \right) &\iff &\alpha_1 = 1 
        \\
\alpha \in  \left[ \frac{2}{10}, \frac{3}{10} \right) &\iff &\alpha_1 = 2
        \\
 \quad \vdots
        \\
\alpha \in  \left[ \frac{9}{10}, 1 \right) &\iff &\alpha_1 = 9.
    \end{array}
\]
For instance,
   \begin{center}
    \begin{tikzpicture}       
         \draw[|-|] (0,0) node [below=1mm] {$0$} -- (10,0) node [below=1mm] {$1$};

        \foreach \x in {1,2}
            \draw (\x,0.08)--(\x,-0.08) node [below=1mm] {\small $\frac{\x}{10}$};

        \foreach \x in {5,...,9}
            \draw (\x,0.08)--(\x,-0.08) node [below=1mm] { \small $\frac{\x}{10}$};

        \filldraw [black] (3.25, 0) circle (1pt) node [above=1mm] { $\alpha=0.3\ldots$}; 

        \foreach \x in {3, 4}
            \draw [very thick, teal] (\x,0.1)--(\x,-0.1) node [below=1mm] {$\mathbf{\frac{\x}{10}}$};
    \end{tikzpicture} 
   \end{center}
Recall that for rational numbers having two expansions  we commit to the expansion ending with 
an infinite  tail of $0$s instead of the infinite  tail of  $9$'s.
This is justified because  we are working with intervals that include the left point,
For example, $\alpha = 0.30000 = 0.2999 \ldots$  and  
  $\alpha \in [3/10, 4/10)$,  so it is justified that we chose the expansion that sets $\alpha_1 = 3$.

If  instead of partitioning $[0,1)$ in $10$ intervals, 
we partition it in  $100$ intervals of size $1/10^2$, 
then to find out the subinterval containing  $\alpha$ we need to determine
the first two digits of its decimal expansion, $\alpha_1$ and $\alpha_2$. For instance,
    \[
    \alpha \in \left[\frac{32}{10^2}, \frac{33}{10^2} \right) \iff \alpha_1 = 3 \: \text{ and } \: \alpha_2 = 2. 
    \]
    \begin{center} 
    \begin{tikzpicture}
        
        \draw[|-|] (0,0) node [below=1mm] {$0$} -- (10,0) node [below=1mm] {$1$};

        \foreach [evaluate=\y using 0.1*\x] \x in {1,...,99}
            \draw (\y,0.08)--(\y,-0.08);

        \filldraw [black] (3.25, 0) circle (1pt) node [above=1mm] {$\alpha=0.32\ldots$};

        \foreach [evaluate=\y using 0.1*\x] \x in {30, 40}
            \draw [very thick, teal] (\y,0.1)--(\y,-0.1) node [below=1mm] {\small $\mathbf{\frac{\x}{10^2}}$};

        \foreach [evaluate=\y using 0.1*\x] \x in {32, 33}
            \draw [very thick, orange] (\y,0.1)--(\y,-0.1);
    \end{tikzpicture}
    \end{center}
For each $n \in \{1, \ldots, N\}$, we need to identify the interval  of size of size $1/10^{n}$ that contains 
 $x_n = 10^{n-1} \alpha \mod 1$, 
 For instance, partition  the interval $[0,1)$ into $1000 $ 
intervals of size $1/10^3$ and consider one of them,
    \[
    I = \left[\frac{325}{10^3}, \frac{326}{10^3}\right).
    \]
    Since $x_n = 10^{n-1} \alpha \mod 1 = 0.\alpha_n \: \alpha_{n+1} \: \alpha_{n+2} \ldots$, we have that
    \begin{align*}
    x_n \in I & \iff \alpha_n = 1 \text{ and } \alpha_{n+1} = 3 \text{ and } \alpha_{n+2} = 0 
\\& \iff \alpha = 0. \alpha_1 \ldots \alpha_{n-1} \: 3 \: 2 \: 5 \: \alpha_{n+3} \ldots.
    \end{align*}
    \begin{center}
    \begin{tikzpicture} 
        
        \draw[-] (0,0) -- (10,0);

        \draw [very thick, teal] (0,0.1)--(0,-0.1) node [below=1mm] {\small $\mathbf{\frac{3}{10}}$};

        \draw [very thick, teal] (10,0.1)--(10,-0.1) node [below=1mm] {\small $\mathbf{\frac{4}{10}}$};

        \draw [very thick, orange] (2,0.1)--(2,-0.1) node [below=1mm] {\small $\mathbf{\frac{32}{10^2}}$};

        \draw [very thick, orange] (5,0.1)--(5,-0.1) node [below=1mm] {\small $\mathbf{\frac{33}{10^2}}$};

        \filldraw [black] (3.25, 0) circle (1.5pt) node [above=1mm] { $x_n=0.325\ldots$};

        \draw [very thick, violet] (3,0.1)--(3,-0.1) node [below=1mm] {\small $\mathbf{\frac{325}{10^3}}$ };

        \draw [very thick, violet] (4,0.1)--(4,-0.1) node [below=1mm] {\small $\mathbf{\frac{326}{10^3}}$ };
    \end{tikzpicture}
    \end{center}
    That is, $x_n$ belongs to the interval $I$ if and only if
 there is an occurrence of the block $B = (3 \: 2 \: 5)$ in the digit number $n$ of $\alpha$. 
Thus, if we take $k = 3$, the length of the block $B$, then
    \begin{align*}
        \# \{ n :  x_n\in I, 1\leq n\leq N\} 
        &= \# \{ n : (\alpha_n \: \alpha_{n+1} \: \alpha_{n+2})
        = (3\:2\:5), 1\leq n\leq N \} \\
        &= \Nu(\alpha,B, N+k-1).
    \end{align*}\flushright{\qedhere}
 \end{example*}

We use Landau's notation to make estimates.  
For functions  $f$ and $g$ over the real numbers, and $g$ strictly positive, we write $f(x)=O(g(x))$ if there exists a positive constant~$C$ and a value~$x_0$ such that for all $x>x_0$,
$|f(x)|<Cg(x)$. And we write
$f(x)=o(g(x))$
if for all positive constants $C$, there exists $x_0$ such that for every $x\geq x_0$,$|f(x)|<Cg(x)$. 

Schiffer in~\cite{Schiffer}  gives  the discrepancy of  numbers in a large family,
 but of  Lebesgue measure zero. 
He proves that for any non-constant polynomial $f$ with 
rational coefficients such that $f(n) \in \N$ for all $n \in \N$, the discrepancy of  the real number $\alpha$ 
whose decimal expansion is formed by the concatenation of the values of $f$ 
evaluated on the positive integers, 
\[
\alpha = 0.f(1)f(2)f(3)\ldots
\]
satisfies the following:
There are two constants $K_1$ and $K_2$ such that, 
there are cofinitely 
many $N$ for which  $D(\alpha, N) < K_1/{\log N}$; 
and there are  infinitely many $N$ for which $D(\alpha, N) > {K_2}/{\log N}$.
Nakai and Shiokawa in~\cite{NS} generalize this result for   non-constant $f$ with real coefficients, such that $f(t)>0$ for all $t>0$.
Schiffer's result applies to the Champernowne constant taking the polynomial $f(x) = x$.
In this note we give  a discrete an elementary proof of of the exact discrepancy of the Champernowne number.

How  does it compare  the discrepancy of the Champernowne constant to  the discrepancy of other normal numbers?
The  minimum discrepancy achievable by a normal number is still not known.
The question was posed by Korobov in 1955~\cite{korobov}, see also Bugeaud's book \cite{Bugeaud}.
Without restricting to sequences of the form  $(b^n \alpha)_{n\geq1}$ the minimal discrepancy known:
Schmidt \cite{schmidt} proved  that  there is a constant $C$ such that for {\em all} sequences $(x_n)_{n\in\N}\subset [0,1)$ there are infinitely many $N$  where the discrepancy of the first $N$ terms, is above $C\log(N)/N$.
And the discrepancy of  first $N$ terms of the van der Corput sequences is $O((\log N)/N)$  so   this is  the minimal discrepancy that an arbitrary sequence $(x_n)_{n\geq 1}$ can have \cite{KN}.

 Surprisingly, almost all numbers in the sense of the Lebesgue measure are normal  with the same discrepancy. 
Gál and Gál~\cite{gal} proved an upper bound for the discrepancy 
of almost all numbers with respect to 
Lebesgue's measure. 
Philipp~\cite{Philipp} gave the explicit constants, 
and Fukuyama~\cite{fukuyama} refined them, 
obtaining that for every base  $b > 1$, 
there exists a constant $K_b$ such that for almost every real number $\alpha$,
\[
\limsup_{N \to \infty} \frac{\sqrt{N} \: D\left((b^n \alpha \mod 1)_{n \geq 0} \: , N \right)}{\sqrt{\log(\log N)}} = K_b.
\]
That is to say that for almost every real number $\alpha$,
\[
D(\alpha, N)=O \left(\frac{\sqrt{\log(\log N)}}{\sqrt{N}} \right).
\]

 Gál and Gál showed that 
 the discrepancy of almost all numbers is below 
 the law of the iterated logarithm. This, in turn, is below the  of discrepancy of  Schiffer's numbers. 
Thus, Champernowne constant approaches  normality  much slower than almost every number.

The set of bases to which a real number can be normal is not tied to any arithmetical properties other than multiplicative dependence 
(for integer bases $r$ and $s$, if $ r^n = s^m $ for some $m,n\in\N$,   a number is normal in base $r$  exactly when  it is normal to base $s$).
For any given set of bases closed under multiplicative dependence, there are real numbers that are normal to each base in the given set, but not  normal to any base in its complement. In \cite[Theorem 2.8]{BS} Becher and Slaman show that the  discrepancy functions for multiplicatively independent bases are pairwise independent.

\section{Statement of results}

Here we  give a  discrete and elementary proof 
of the exact discrepancy of the Champernowne  constant $c_{10}$.
Schiffer \cite{Schiffer} provided the discrepancy of a family of numbers that includes~$c_{10}$. His  proof relies on exponential sums. 
 This is and alternative  proof of   Schiffer's result  specifically for the  Champernowne constant.

\begin{theorem} \label{thm:1}
Let $c_{10}$ be the Champernowne constant for base $10$.
Then, there is a constant $K_1 > 0$ such that there are cofinitely many N such that
$
D(c_{10}, N) < K_1/{\log N}.
$
\end{theorem}

\begin{theorem} \label{thm:2}
 Let $c_{10}$ be the Champernowne constant for base $10$. Then, there exists a constant $K_2 >0$ such that for infinitely many~$N$,
$
 D(c_{10}, N) > {K_2}/{\log N}.
$
In particular,  one can take  $K_2 = 1/( 10^3 3)$.
\end{theorem}

\begin{remark*}
 Theorems \ref{thm:1} and \ref{thm:2} imply that $D(c_{10}, N) = O \left({1}/{\log N} \right)$ and $D(c_{10}, N) \neq o \left({1}/{\log N}\right)$. 
That is, the estimate cannot be improved.
\end{remark*}

\begin{remark*}
Theorems~\ref{thm:1} and~\ref{thm:2} 
hold for any other integer base $b\geq 2$ with the 
definition of  Champernowne constant $c_b$ 
for that base (the concatenation of all positive integers represented in that base).  
The constants $K_1,K_2$ depend on the base $b$.
\end{remark*}

\section{Basic tools}

In the sequel the Champernowne constant is called $c=c_{10}$,
\[
c=0.12345678910111213141516171819202122\ldots
\]
We define three sequences.

\begin{definition*} 
 Let  $(t_i)_{i \in \N}$, where $t_i = i$, be the sequence of terms that concatenated  yield the expansion of $c$.
\end{definition*}

\begin{definition*} 
Let   $(c_i)_{i\in\N}$,  where each $c_i\in\{0,..,9\}$, 
be the decimal expansion of $c$.
\end{definition*}

For example, $c_{11} = 0$ and $c_{14} = 1$, because
 \[ c = 0. \! \underset{\substack{\uparrow \\ c_1}}{1} \overset{\substack{c_2 \\ \downarrow}}{2} \: 3 \: 4 \: 5 \: 6 \: 7 \: 8 \:9 \underset{\substack{\uparrow \\ c_{10}}}{1} \! \! \! \overset{\substack{c_{11} \\ \downarrow}}{0} 11 \underset{\substack{\uparrow \\ c_{14}}}{1} \! 2 \: 13 \: 14 \:15 \:16 \ldots.
   \] 

\begin{definition*} 
Let  $(x_n)_{n \in \N} = (10^{n-1} c \mod 1)_{n \in \N}$. 
\end{definition*}
Thus,
$        x_1 = 0. 1 \: 2 \: 3 \: 4 \: 5 \ldots ,
        x_2 = 0. 2 \: 3 \: 4  \:5 \: 6 \ldots,
\ldots 
        x_n = 0. c_n \: c_{n+1} \: c_{n+2} \ldots
$

\begin{definition*}[Overlapping  occurrences]
An occurrence of $B$ in $c$ is \emph{overlapping} if $B$ occurs between two or more $t_i$.
 We denote $\Nu_o(c, B, N)$ to the number of overlapping occurrences of $B$ in $(c_1 \: c_2 \ldots c_N)$.
\end{definition*}

\begin{definition*}[Non-overlapping  occurrences]
An occurrence of $B$ in $c$ \emph{non-overlapping}, if $B$ occurs within a single term $t_i$.
 We denote $\Nu_{no}(c, B, N)$ to the number of non-overlapping occurrences of $B$ in $(c_1 \: c_2 \ldots c_N)$.
\end{definition*}

\begin{example*}
    For $N = 36$ and $B = (1\:2)$. Then, 
    $\Nu_{no}(c, B, N) = 1$, because 
    \[
    c = 0. 1 \: 2 \: 3 \: 4 \: 5 \: 6 \: 7 \: 8 \: 9 \:10 \:11 \: \mathbf{12} \: 13 \: 14 \: 15 \: 16 \: 17 \: 18 \: 19 \: 20 \: 21 \: 22 \underset{\substack{\uparrow \\ x_N}}{2}\!3 \: 24 \ldots
    \]
And    $\Nu_o(c, B, N) = 2$, because
    \[
    c = 0.\mathbf{1 \: 2} \: 3 \: 4 \: 5 \: 6 \: 7 \: 8 \: 9 \:10 \:11 \: 12 \: 13 \: 14 \: 15 \: 16 \: 17 \: 18 \: 19 \: 20 \: 2\mathbf{1 \: 2} 2 \underset{\substack{\uparrow \\ x_N}}{2}\!3 \: 24 \ldots.
    \]
\flushright{\qedhere}\end{example*}

\begin{remark*}
    $\Nu(c,B, N) = \Nu_{no}(c, B, N) + \Nu_o(c, B, N)$.
\end{remark*}

\begin{definition*}[segment $s_\ell$]
    Given $\ell \in \N$,  $s_\ell$ is 
the concatenation of terms  
formed by all natural numbers of $\ell$ digits ordered in ascending order, that is, from $10^{\ell - 1}$ to $10^{\ell} - 1$,
\[
s_\ell=(\overbrace{10\ldots0}^{\ell \text{ digits}}, \: \ldots \: , \overbrace{9\ldots9}^{\ell \text{ digits}}).
\]
\end{definition*}

\begin{example*}
    $s_1 = (1 \: 2 \: 3 \: 4 \: 5 \: 6 \: 7 \: 8 \: 9)$,
    $s_2 = (10 \: 11 \: 12 \ldots 98 \: 99)$,
    $s_3 = (100 \: 101 \ldots 998 \: 999)$.
\end{example*}

\begin{definition*}[The numbers $v=v(N)$ and $n=n(N)$]
For  $N \in \N$, let $v=v(N)$ be such that $T(v) \geq N$ and $T(v-1) < N$. 
Let  $n=n(N)$ the number of digits of $v$.
\end{definition*}

\begin{definition*}
For  $v \in \N$,  $T(v)$ is the number of digits in the  expansion of $c$ up to the term $v$.
\end{definition*}

\begin{example*}
For $v = 11$,  $T(v) = 13$, since 
     there are $13$ digits in $ 1 \: 2 \: 3 \: 4 \: 5 \: 6 \: 7 \: 8 \: 9 \: 10 \: 11$.
\end{example*}

\begin{example*}
     For $N = 14$, then $v = 12$ and $n = 2$, since $T(v) = 15 \geq N$ and $T(v-1) = 13 < N$.
     \[
     c = 0. 1\:2\:3\:4\:5\:6\:7\:8\:9 \:10\:11 \! \overbracket{\underset{\substack{\uparrow \\ c_N}}{1} \! 2}^{v} 13\ldots 
     \]\hfill{\qedhere}
\end{example*}

\begin{definition*}[Lengths $L$ and $M$]
Given $N$, 
we define $L=L(N)$ as  number of digits from $s_1$ to $s_{n-1}$,
    \[
     L = \sum _{j=1}^{n-1} j \cdot 9 \cdot 10^{j-1}.
     \]
We define  $M=M(N)$ 
as  the number of digits within $s_n$ to the number $v$.
     \[
     M = n(v - 10^{n-1} + 1) = n \left( \sum \limits_{i=1}^n v_i 10^{n-i} - 10^{n-1} + 1 \right ).
     \]
 \end{definition*}

\begin{remark*}
Given $n\in \N$, let $L=L(N),M=M(N), n=n(N)$. Then,
     \begin{align}\label{eq:N}
     L \leq N \leq L + M = \sum _{j=1}^n 9 j 10^{j-1 } = n 10^n - \frac{10^n}{9} + \frac{1}{9}.
     \end{align}
\end{remark*}

    \begin{definition*}
$\Nu_{no}(B, v)$ is the number of non-overlapping occurrences of $B$ in all $n$ digit terms less than or equal to $v$, where $n$ is the number of digits of $v$. 
     \end{definition*}

    \begin{definition*}
$\Nu_{o}(B, v)$ is the number of overlapping occurrences of $B$ in all $n$ digit terms less than or equal to $v$, where $n$ is the number of digits of $v$. 
     \end{definition*}

   \begin{definition*}
For $\ell \in \N$, $\U_{no}(B, s_\ell)$ is the number 
of non-overlapping occurrences of $B$ in $s_{\ell}$. 
     \end{definition*}

          \begin{definition*}
     For$\ell \in \N$,  $\U_o(B, s_\ell)$ is the number of overlapping occurrences in $s_{\ell}$. 
     \end{definition*}

\section{Theorem~\ref{thm:1}: Upper bound}

We  give an upper bound, for every $N$, of $    D(c, N)$.

\subsection{Counting occurrences}
We start with the following lemma.

\begin{lemma} \label{lem:1}
     Given $N \in \N$ and $B$ a block of length $k>1$, then
     \[\Nu(c,B, N) = 10^{-k}N + O(10^{n-k}),
     \]
     where the hidden constant in $O(10^{n-k})$ does not depend on $B$.
\end{lemma}

\begin{proof}
     Let $N \in \N$ and $B = (b_1, \ldots, b_k)$. 
     We know that 
     \[
     \Nu(c,B, N) = \Nu_{no}(c, B, N) + \Nu_o(c, B, N).
     \]
     We separate  the proof into steps. 

     \textbf{Step 1:}
     We estimate $\Nu_{no}(c, B, N)$, the non-overlapping occurrences.  

     \textbf{Step 1.1:} 
For each $\ell\geq 1$ we show 
     \[
     \U_{no}(B, s_\ell) \leq 10^{\ell-k} + (\ell-k)\cdot 9 \cdot 10^{\ell-k-1}.
     \]
     
     If $\ell < k$, then $\U_{no}(B, s_\ell) = 0$, since $B$ does not fit within a block of $\ell$ digits.
     
     If $\ell \geq k$: We want to count the amount of numbers of the form:
     \[
     y = \underbrace{* \ldots * B * \ldots *}_{\ell \text{ digits}}
     \]
     where the asterisk $*$ represents any possible digit. 
     We  move the position of $B$ and count in each case:
 
$\bullet$ We count the amount of numbers of the form
         \[
         y_0 = B \underbrace{* \ldots *}_{\ell-k \: \text{digits}}
         \]
         that is, $\ell$ digit numbers that contain $B$ in the first position: 
         
         If $b_1 = 0$: There are 0 numbers of the form $y_0$, since no natural number begins with 0. 
         
         If $b_1 \neq 0$: There are $10^{\ell-k}$ numbers of the form $y_0$, so we can choose the last $\ell-k$ digits between 0 and 9.

$\bullet$         We count the amount of numbers of the form
         \[
         y_1 = \underbrace {*}_{1 \text{ digit}} \: B \underbrace{* \ldots *}_{\ell-k-1 \text{ digits}}
         \]
         that is, $\ell$ digit numbers that contain $B$ in the second position:
         There are $9 \cdot 10 \cdot 10^{\ell -k-1}$
         numbers of the form $y_1$.

$\bullet$       
We count the amount of numbers of the form
         \[
         y_2 = \underbrace {*\:*}_{2 \text{ digits}} B \underbrace{* \ldots *}_{\ell-k-2 \text{ digits}}
         \]
         that is, $\ell$ digit numbers that contain $B$ in the third position:
         There are $9 \cdot 10 \cdot 10^{\ell -k-2}$ numbers of the form $y_2$. Because:
         \begin{itemize}[label = $\circ$]
             \item $9$ is the number of values that the first digit can take (between $1$ and $9$ since it cannot take the value $0$).
             \item $10$ is the number of values that the second digit can take (between $0$ and 9).
             \item $10^{\ell - k - 2}$ is the number of values that the last $\ell - k - 2$ digits can take.
         \end{itemize}

$\bullet$  We count the amount of numbers of the form
         \[
         y_3 = \underbrace {*\:*\:*}_{3 \text{ digits}} B \underbrace{* \ldots *}_{\ell-k-3 \text{ digits}}
         \]
         that is, $\ell$ digit numbers that contain $B$ in the fourth position:
         There are $9 \cdot 10^2 \cdot 10^{\ell -k-3}$ numbers of the form $y_3$.
 
$\bullet$ 
     Continuing like this, we arrive at the last position:
We count the amount of numbers of the form
         \[
         y_{\ell-k} = \underbrace {*\ldots*}_{\ell-k \text{ digits}} B
         \]
         that is, $\ell$ digit numbers containing $B$ in position $\ell-k+1$ :
         There are $9 \cdot 10^{\ell-k-1}$ numbers of the form $y_{\ell - k}$.
         
Overall we obtain that for each $j \in \{1, \ldots, \ell - k\}$, the amount of numbers of the form $y_j$ is
     \[
     9 \cdot 10^{\ell - k - 1}.
     \]
     Therefore, 
     if $b_1 = 0$: 
     
     $\U_{no}(B, s_\ell) = \sum _{j=1}^{\ell-k} 9 \cdot 10^{\ell - k - 1} = (\ell -k)\cdot 9 \cdot 10^{\ell-k-1}$;
     \\
     if $b_1 \neq 0$:
     
     $\U_{no}(B, s_\ell) = 10^{\ell-k} + \sum _{j=1}^{\ell-k} 9 \cdot 10^{\ ell - k - 1} = 10^{\ell-k} + (\ell-k)\cdot 9 \cdot 10^{\ell-k-1}$. 
     \\
     Then, for all $\ell \geq k$,
     \[
     \U_{no}(B, s_\ell) \leq 10^{\ell-k} + (\ell-k)\cdot 9 \cdot 10^{\ell-k-1},
     \]

     \textbf{Step 1.2:}
We  show that
\[
      \Nu_{no}(c, B, N) = \sum_{\ell = k}^{n-1} \U_{no}(B, s_\ell) + \Nu_{no}(B, v) + O(n).
     \]
     To count non-overlapping occurrences of $B$ up to the position $N$,
     that is, $\Nu_{no}(c, B, N) $), 
we  count occurrences in $s_{\ell }$ for each $\ell \in \{1,\ldots,n-1\}$ that is, $\U_{no}(B, s_\ell)$,
and then count the occurrences in $s_n$ by cutting it at $v$ (that is  $\Nu_{no}(B, v)$).
 Finally, the $O(n)$ comes from substracting the possible occurrences of $B$ within $v$  
after the $N$'th digit, there could be at most $n-k$ of those occurrences. 
\medskip
  
    \begin{example*}
         For $N=9523$,  $v=2658$ and $n=4$,
         \[
         c = 0. \underbrace{1 \ldots 9}_{\substack{9\cdot1 \\ \text{digits}}}
         \underbrace{10 \ldots 99}_{\substack{90\cdot2 \\ \text{digits}}}
         \underbrace{100 \ldots 999}_{\substack{900\cdot3 \\ \text{digits}}}
         \underbrace{1000 \ldots 2658}_{\substack{1658\cdot 4 \\ \text{digits}}}
         \: \overbracket{2 \! \underset{\substack{\uparrow \\ c_N}}{6} \! 58}^{v}
         \]
         Let's take $B$ to be any two-digit block, that is, $k=2$. 
         To count non-overlapping occurrences up to $N$, we have
         \[
         \Nu_{no}(c, B, N) = \underbracket{\U_{no}(B, s_2)}_{\substack{\text{occurrences} \\ \text{in } s_2}} +
         \underbracket{\U_{no}(B, s_3)}_{\substack{\text{occurrences} \\ \text{in } s_3}} +
         \underbracket{\Nu_{no}(B, v)}_{\substack{\text{occurrences} \\ \text{in } s_n \text{ until } v}} \: \: + \! \! \!
         \underbracket{O(n)}_{\substack{\text{occurrences} \\ \text{after } c_N}}
         \]
         because
         \[
         c = 0. \overbracket{1 \ldots 9}^{s_1}
         \overbracket{10 \ldots 99}^{s_2}
         \overbracket{100 \ldots 999}^{s_3}
         1000 \ldots 2658
         \: \overbracket{2 \! \underset{\substack{\uparrow \\ c_N}}{6} \! 58}^{v}.
         \]
          Then, the procedure consists of first, counting the occurrences in $s_1$ 
(which are 0 because $k=2$), in $s_2$ and $s_3$; then count the occurrences
 from $1000$ to $2658$; and finally, subtract possible occurrences in the last two digits of $v$.       
     \end{example*}

     \textbf{Step 1.3:} Bound $\Nu_{no}(B, v)$. 
which  is the amount of numbers of the form
     \[
     y = \underbrace{* \ldots * B * \ldots *}_{n \text{ digits}} \qquad \text{with } y \leq v.
     \]
     
     If $n < k$: $\Nu_{no}(B, v) = 0$, because, $B$ does not fit inside blocks of length $n$.
   
     If $n \geq k$: 
     Let $v = \sum \limits_{i=1}^{n} v_i 10^{n-i} = v_1 \ldots v_n$.
     Given $j \in \{0, \ldots, n-k\}$, we define: 
     \[
     a_j = \sum \limits_{i=1}^{j} v_i 10^{j-i} = v_1 \ldots v_j\]
     Again, we  move the position of $B$ and count in each case: 
     Given $j \in \{0, \ldots, n-k\}$, let us call $\Nu_{no}(B, v, j)$ the amount of numbers of the form
     \[
     y_j = \underbrace{*\ldots*}_{j \text{ digits}} \: B \! \underbrace{*\ldots*}_{n - k - j \text{ digits}} \quad \text{with } y_j \leq v.
     \]
     
     {\em Case ${j=0}$}. We want to count the amount of numbers of the form
     \[
     y_0 = B \underbrace{*\dots*}_{n-k \: \text{digits}}
     \]
     We note that again, if $b_1 = 0$, then $\Nu_{no}(B, v, j) = 0$, since there are no numbers that begin with a leading zero. Now, if $b_1 \neq 0$, we must separate cases, since $\Nu_{no}(B, v, j)$  depend on who is $B$ and who is $v$. \\
     \begin{itemize}
         \item If $B > v_1 \dots v_k$: $\Nu_{no}(B, v, j) = 0$, then $y_0 > v$.
         \item If $B = v_1 \dots v_k$: $\Nu_{no}(B, v, j) = v_{k+1} \dots v_{n} + 1$, then the last $n-k$ digits of $y_0$ we can choose between 0 and $v_{k+1} \dots v_{n}$.
         \item If $B < v_1 \dots v_k$: $\Nu_{no}(B, v, j) = 10^{n-k}$, then the last $n-k$ digits of $y_0$ can take any value from 0 up to $\underbrace{9\dots9}_{n-k \: \text{digits}}$.
     \end{itemize}
     Therefore, $\Nu_{no}(B, v, 0) \leq 10^{n-k}$. \bigskip

{\em Case ${1 \leq j \leq n-k}$
}.
\begin{align*}
     \Nu_{no}(B, v, j)& = \underbracket{(a_j - 10^{j-1} + \theta_j)}_{\substack{\text{Choose the first} \\ j \text{ digits } }} \underbracket{10^{n-k-j}}_{\substack{\text{Choose the last} \\ n-k-j \text{ digits}}} 
     \\
     &= 10^{-k} \left( \sum _{i=1 }^{j} v_i 10^{n-i} - 10^{n-1} + \theta_j10^{n-j} \right)  
\text{     where $0 \leq \theta_j \leq 1$.}
     \end{align*}.
     
     Let's see  why the first equality is valid.
     For $y_j$ to be effectively less than or equal to $v$, we must choose the first $j$ digits within $(10^{j-1}, \ldots, a_j)$. It depends on the values of $B$ and $v$, whether we are taking $a_j$ inclusive or exclusive.
 
    \begin{example*}
         $B = ({\color{purple}3 \: 1})$, $n = 4$, $v = 2{\color{purple}\mathbf{32}}5$ and $j= 1$.
         So, 
\[
\Nu_{no}(B, v, 1) = \underbracket{2}_{\substack{\text{Choose the}
 \text{first digit} \\ \text{between 1 and 2} }} \cdot
 \underbracket{10}_{\substack{\text{Choose the}\\ \text{fourth digit} 
 \text{between 0 and 9}}} 
= (2 - 10^{1-1 } + \theta_1) 10^{4-2-1}
\]
 with $\theta_1 = 1$.
         In this case, we are taking $a_j = 2$ inclusive.
     \end{example*}
    \medskip

     \begin{example*}
         $B = ({\color{purple}3\:1})$, $n = 4$, $v = 2{\color{purple}\mathbf{30}}5$ and $j= 1$.
         Then, 
         \[
         \Nu_{no}(B, v, 1) = \underbracket{1}_{\substack{\text{The first}\\ \text{digit only} \\ \text{can be 1}} } \cdot \underbracket{10}_{\substack{\text{Choose the}\\ \text{fourth digit} \\ \text{between 0 and 9}}} = (2 - 10^{1-1} + \theta_1) 10^{4-2-1}, \text{          with $\theta_1 = 0$.}
         \]
\text{          with $\theta_1 = 0$.}
         In this case, we are taking $a_j = 2$ exclusive.
     \end{example*}
        \medskip

     \begin{example*}
         $B = ({\color{purple}3\:1})$, $n = 4$, $v = 2{\color{purple}\mathbf{31}}5$ and $j= 1$.\ \
         So,
         \begin{align*}
             \Nu_{no}(B, v, 1)
             &= \underbracket{10}_{\substack{\text{If the first}\\ \text{digit is 1,} \\ \text{choose the fourth}\\ \text{digit between 0 and 9}} } + \underbracket{6}_{\substack{\text{If the first}\\ \text{digit is 2,} \\ \text{choose the fourth}\\ \text{digit between 0 and 5}} }
             \\
             &= (2 - 10^{1-1}) 10^{4-2-1} + (10^{4-2-1} - 4)
             \\
             &= (2 - 10^{1-1}) 10^{4-2-1} + 10^{4-2-1}(1-\frac{4}{10})
             \\
             &= (2 - 10^{1-1} + \theta_1) 10^{4-2-1},
\text{     with $\theta_1 = 1-\frac{4}{10}$.}
         \end{align*}\flushright{}
     \end{example*}
    
Recall $M$
is the number of digits from $10^{n-1}$ to $v$, that is, the number of digits in $s_n $ up to the number $v$. Then,
 Recall  $L$ the number of digits from 1 to $10^{n-1}-1$
 By \eqref{eq:N}, $L\leq N\leq L+M$.     Indeed, $N = L + M - O(n)$ because in the worst case, $N$ is the position of the first digit of $v$ and we have to subtract $n-1$.
We have,
\setlength{\mathindent}{0pt}

\begin{align*}
             \Nu_{no}(B, v)
             &= \sum_{j=0}^{n-k} \Nu_{no}(B, v, j) 
             \\
             &\leq 10^{n-k} + 10^{-k} \left( \sum _{j=0}^{n-k} \sum _{i=1}^j v_i 10^{n-i} - 10^{n- 1} + \theta_j 10^{n-j} \right) \notag
             \\
             &= 10^{-k} \left( \sum _{j=0}^{n-k} \sum _{i=1}^j (v_i 10^{n-i} - 10^{n-1}) + \sum_ {j=0}^{n-k} \sum _{i=1}^j \theta_j 10^{n-j} \right) + O(10^{n-k}) \notag 
             \\
             &\leq 10^{-k} \left( -(n - k + 1) 10^{n-1} + \sum _{j=0}^{n-k} \sum _{i=1}^j v_i 10 ^{n-i} \right) + 10^{-k} 10^n \sum _{j=0}^{n-k} j \frac{1}{10^j} + O(10^{n-k}) \notag
             \\
             &= 10^{-k} \left( -(n - k + 1) 10^{n-1} + \sum _{j=0}^{n-k} \sum _{i=1}^j v_i 10^ {n-i} \right) + O(10^n) + O(10^{n-k}) \notag
             \\
             &\leq 10^{-k} \left( -(n - k + 1) 10^{n-1} + \sum _{j=1}^{n-k} v_i 10^{n-i} (n-k-i+ 1) \right) + O(10^{n-k}) 
             \\
             &\leq 10^{-k} M + O(10^{n-k}) 
         \end{align*}
     Let's see why the last two inequalities are valid.
  For  the previous to the last we have
      \begin{equation*}
          \begin{split}
              \sum _{j=0}^{n-k} \sum _{i=1}^j v_i 10^{n-i} &= \sum _{i=1}^0 v_i 10^{n-i} + \sum _{i=1 }^1 v_i 10^{n-i} + \ldots \sum _{i=1}^{n-k} v_i 10^{n-i} \\
              &= (v_1 10^{n-1}) + (v_1 10^{n-1} + v_2 10^{n-2}) + \ldots + (v_1 10^{n-1} + \ldots + v_ {n-k} 10^{n-(n-k)}) \\
              &= \sum _{i=1}^{n-k} v_i 10^{n-i} (n-k-i+1)
          \end{split}
      \end{equation*}
 For the last we have,
      \begin{align*}
      -(n - k + 1) 10^{n-1} + \sum _{j=1}^{n-k} v_i 10^{n-i} (n-k-i+1) &\leq -n10^{n-1} + 10^{n-1}(k-1) + n\sum _{i=1}^n v_i 10^{n-i} 
      \\&\leq M + O(10^n).
      \end{align*}
We conclude,
     \begin{equation}  \label{eq:2}
         \begin{split}
             \Nu_{no}(c, B, N) &= \sum_{\ell = k}^{n-1} \U_{no}(B, s_\ell) + \Nu_{no}(B, v ) - O(n) \\
             &\leq \sum_{\ell = k}^{n-1} 10^{\ell - k} + (\ell - k) \cdot 9 \cdot10^{\ell-k-1} + 10^{ -k} M + O(10^{n-k}),\\
         \end{split}
     \end{equation}
     where the hidden constant inside $O(10^{n-k})$ does not depend on $B$.
     \medskip

\setlength{\mathindent}{35pt}

     \textbf{Step 2:} We bound $\Nu_o(c, B, N)$, the number of overlapping occurrences. 

If $k = 1$, there are no overlapping occurrences of $B$, so we assume $k > 1$. 
Observe that 
     \[
     \Nu_o(c, B, N) \leq \sum \limits_{\ell = 1}^{n} \U_o(B, s_\ell) + O(n).
     \]
The worst case for  $N$ is realized in the last digit of $v$ when  $v$ is the last number in $s_n$,
\[
v = 10^{n}-1 = \underbrace{9 \ldots 9}_{n \text { digits}},
\]
so  should add each $\U_o(B, s_\ell)$ up to $\ell = n$. The $O(n)$ comes from summing all the overlapping occurrences that could appear between two (or more) $s_\ell$. At most there are $kn$ of those occurrences, that is $k$
 occurrences for each $s_\ell$). We can bound them by a constant that does not depend on the choice of~$B$.
\medskip    
\pagebreak

     \textbf{Step 2.1:} We bound $\U_o(B, s_\ell)$. 

     {\em Case  ${\ell \geq k}$}. For blocks of length $\ell \geq k$, there can only be overlapping occurrences two blocks, and no more.
     If $x$ is a number of $\ell$ digits, and overlapping occurrence of $B$ between $x$ and $x+1$ can happen only if the last digits of $x$ are $(b_1, \ldots, b_{k-j})$ and the first digits of $x+1$ are $(b_{k-j+1}, \ldots, b_k)$ for some $j \in \{ 1, \ldots , k-1 \}$. Therefore, $x$ must be of the form
     \[
     x = b_{k-j+1} \ldots b_k \underbrace{* \ldots *}_{\ell - k \text{ digits}} b_1 \ldots b_{k-j}.
     \]
Since there are $\ell - k$ free digits, $x$ can at most take $10^{\ell-k}$ values. Then, for each
     \\$j \in \{ 1, \ldots, k-1 \}$, $B$ can occur overlapping blocks of length $\ell$ at most $10^{\ell - k}$ times. Therefore,
     \[
     \U_o(B, s_\ell) \leq \sum _{j=1}^{k-1} 10^{\ell - k} = (k-1)10^{\ell - k}.
     \]
     \begin{example*} For
         $\ell = 6$, $B = (1\:2\:3\:4)$, $k = 4$. Then, the overlapping occurrences of $B$ are: 
        \medskip

         \begin{tabular}{c |c}
             $x$ & $x+1$\\
         \hline
             $2\:3\:4\:*\:*\:\mathbf{1}$ & $\mathbf{2\:3\:4}\:*\:*\:2$ \\
             $3\:4\:*\:*\:\mathbf{1\:2}$ & $\mathbf{3\:4}\:*\:*\:1\:3$\\
             $4\:*\:*\:\mathbf{1\:2\:3}$ & $\mathbf{4}\:*\:*\:1\:2\:4$\\
         \end{tabular}\\
         \medskip

\noindent
Then, $\U_o(B, s_\ell) = 3 \cdot 10^2$.       
In this  example equality applies because the block $B$ consists of  all different digits. But, if $B$ has repeated digits, we could be counting the same occurrence repeatedly. That's why we get a bound for $\U_o(B, s_\ell)$ and not an equality. 
     \end{example*}

     \begin{example*} For
         $\ell = 6$, $B = (1\:1\:1\:1)$, $k = 4$,  the overlapping occurrences of $B$ are: 
\medskip

         \begin{tabular}{c| c}
             $x$ & $x+1$\\
             \hline
             $1\:1\:1\:*\:*\:\mathbf{1}$ & $\mathbf{1\:1\:1}\:*\:*\:2$ \\
             $1\:1\:*\:*\:\mathbf{1\:1}$ & $\mathbf{1\:1}\:*\:*\:1\:2$\\
             $1\:*\:*\:\mathbf{1\:1\:1}$ & $\mathbf{1}\:*\:*\:1\:1\:2$\\
         \end{tabular}\\ 
         \medskip

        \noindent
 Then, $\U_o(B, s_\ell) < 3\cdot10^2$.
         The equality would imply  that numbers such as   $x = 111011$ be counted  twice  of  (once in the first row and once once in the second row).
     \end{example*}

    {\em Case $\ell < k$}. To simplify the work, and as it is sufficient for the bound we are looking for, we  bound all overlapping occurrences $B$ from $s_1$ to $ s_{k-1}$ in terms of  the number of digits from $s_1$ to $s_{k-1}$, that is, the number of digits in
     \[
     1 \: 2 \ldots 10\:11 \ldots 100 \ldots 999 \ldots \underbracket{10^{k-1}-1}_{\substack{= 9\ldots9 \\ \text{(It has }k- 1 \text{ digits})}}.
     \]
     Then,
     \[
     \sum _{\ell = 1}^{k-1} \U_o(B, s_\ell) \leq \sum _{i=1}^{k-1} \underbracket{9 \cdot 10^{i-1}} _{\substack{\text{Amount of} \\ \text{numbers in } s_i}} \cdot \underbracket{i}_{\substack{\text{Amount} \\ \text{of} \\ \text {digits}}}.
     \]
     Therefore, we bound the overlapping occurrences,  concluding Step 2:
     \begin{equation}  \label{eq:3}
             \Nu_o(c, B, N)
             \leq \sum _{\ell = k}^{n} 10^{\ell - k} (k-1) + \sum _{j=1}^{k-1}9 \cdot 10^{j-1 } \cdot j  
     \end{equation}
Using the bounds $(\ref{eq:2})$ (from Step 1) and $(\ref{eq:3})$ ( from Step 2) we obtain,
     \begin{align*}
             \Nu(c,B, N)
             &= \Nu_{no}(c, B, N) + \Nu_o(c, B, N)
             \\
             &\leq \sum_{\ell = k}^{n-1} 10^{\ell - k} + 9(\ell - k)10^{\ell - k -1} + 10^{n-k} + 10^{-k}M + O(10^{n-k}) 
             \\&\ \ \ + \sum _{\ell = k}^{n} 10^{\ell - k} (k-1) + \sum _{j=1} ^{k-1}9 \cdot 10^{j-1} \cdot j
             \\
             &= \sum_{\ell = k}^{n-1} 10^{\ell - k} + 9(\ell - k)10^{\ell - k -1} + 10^{-k}M 
             \\&\ \ \ + \sum _{\ell = k}^{n} 10^{\ell - k} (k-1) + O(10^{n-k}).
\end{align*}
Then,

\setlength{\mathindent}{0pt}

\begin{align*}
\Nu(c,B, N)             &\leq \sum_{\ell =k}^{n-1} 10^{\ell-k-1}(10 + 9\ell - 9k + 10k - 10) 
+ 10^{n-k}(k-1 ) + 10^{-k}M + O(10^{n-k})
             \\
             &= 10^{-k} \sum _{\ell = k}^{n-1} 10^{\ell - 1} \cdot 9 \cdot \ell + 10^{-k} \cdot k \cdot \frac{1}{10} \sum _{\ell = k}^{n-1} 10^\ell
  O(10^{n-k}) + 10^{-k}M + O(10^{n-k} )
             \\
             &\leq 10^{-k}L + 10^{-k} \cdot k \cdot \frac{1}{10} \left(\frac{1-10^n}{-9} - \frac{ 1-10^k}{-9} \right) 
  + O(10^{n-k}) + 10^{-k}M + O(10^{n-k})
             \\
             &= 10^{-k}L + O(10^{n-k}) + 10^{-k}M + O(10^{n-k})
             \\
             &= 10^{-k}N + O(10^{n-k}).
     \qedhere
     \end{align*}
\end{proof}
\setlength{\mathindent}{35pt}

\subsection{Proof of Theorem~\ref{thm:1}}

We need to show that
     there exists $N_0 \in \N$ and $C > 0$ such that for all $N \geq N_0$,
     $D(c, N) \leq C \frac{1}{\log N}$. 
We introduce notation. 
For $0 \leq a < b < 1$ and $N \in \N$, 
     \[
     D(c,a, b, N) = \frac{\# \{ j \in \{1, \ldots, N\}: x_j\in[0,b)\}}{N} - (b-a).
     \]
     Then,
     \[
     \sup_{0 \leq a < b < 1} |D(c,a, b, N)| = D(c, N),
     \]
To bound the discrepancy of $c$ we should take supremum of $|D(c,a, b, N)|$ over all intervals $[a, b) \subseteq [0,1)$. 
     Let us see that it suffices to consider intervals of the form $[0, b)$.
     Suppose we have proved it for every interval of the form $[0, b)$, that is, for every $ N \gg 1$,
     \[ \sup _{0 \leq b < 1} |D(c,0,b, N)| = \sup _{0<b<1} 
\left| \frac{\# \{ j \in \{1, \ldots, N\}: x_j\in[0,b)\}}{N} - b \right| \leq \frac{K}{\log N}.
     \]
For every $ N \gg 1$, $ |D(c,a, b, N)| $ is equal to 
\[
            \left| \frac{\# \{ j \in \{1, \ldots, N\}: x_n\in[0,b)\} - \# \{ j \in \{1, \ldots, N\}: x_j\in[0,a)\}}{N} - b + a\right|. 
\]
Then,     
\setlength{\mathindent}{0pt}

\begin{align*}     
|D(c,a, b, N)| 
             &\leq \left| \frac{\# \{ j \in \{1, \ldots, N\}: x_n\in[0,b)\}}{N} - b \right|
              + \left| \frac{\# \{ j \in \{1, \ldots, N\}: x_j\in[0,a)\}}{N} - a \right| \\
             &\leq 2 \sup_{0 \leq y < 1} |D(c,0,y, N)| \\
             &\leq 2 \frac{K}{\log N}\\
             &= \frac{C}{\log N}.
         \end{align*}

     \setlength{\mathindent}{35pt}

We now  prove  that for every $ N \gg 1$, for all $b \in [0,1)$,
$\sup_{0 \leq b < 1} |D(c,0, b, N)| \leq \frac{K}{\log N }$,  where the constant $K$ does not depend on $b$. 
It  is enough to see that $|D(c,0, b, N)| = O \left(\frac{1}{\log N } \right)$, for all $b \in [0,1)$, where the hidden constant in Landau's $O$ does not depend on $b$. 
We divide the proof in three steps. 

     \textbf{Step 1:} Let $k \in \N$ and $\alpha \in [0, 1)$, $\alpha = 0.\alpha_1 \ldots \alpha_k = \sum _{i=1}^k \alpha_i 10^{-i}
     $. Let $N \in \N$. We show 
     \[
     |D(c,\alpha, \alpha + 10^{-k}, N)| = O\left(\frac{10^{-k}}{\log N}\right),
     \]
     where the constant does not depend on $\alpha$. 
     For $B = (\alpha_1, \ldots, \alpha_k)$ we have,
     \begin{align*}
     \# \{ j \in \{1, \ldots, N\}: x_j\in [\alpha, \alpha + 10^{-k})\} =
  \Nu(c,B, 1, N) + O(k).
     \end{align*}
The term  $O(k)$ comes from counting the possible occurrences that could occur after the digit $c_N$ and up to the digit $c_{N+k-1}$. 
Using Lemma $\ref{lem:1}$ for the second equality we obtain,
         \begin{align}
             |D(c,\alpha, \alpha + 10^{-k}, N)|
             &= \left| \frac{1}{N} \Nu(c,B, N) - 10^{-k} + O \left( \frac{k}{N} \right) \right| \notag \\
             \notag \\
             &= \left| \frac{1}{N} 10^{-k} N + O \left( \frac{1}{N} 10^{n-k} \right) - 10^{-k} + O \left( \frac {k}{N} \right) \right| \notag \\
              \notag \\
             &= \left| 10^{-k} + O \left( \frac{1}{\log N} 10^{-k} \right) - 10^{-k} + O \left( \frac{k}{N} \right) \right| \label{eq:4} 
             \\
             \notag \\
             &= O \left( \frac{1}{\log N} 10^{-k} \right), \notag
             \end{align}
    for $N$ large enough (since $k$ is fixed). 
    By Lemma \ref{lem:1}, the hidden constant in $ O \left( \frac{10^{-k}}{\log N}\right)$ does not depend on $\alpha$. 

Recall that 
$N$ is the position at which we find $v$, 
$v$ is a term with $n$ digits,
$L$ is  number of digits from $s_1$ to $s_{n-1}$, and $M$ is the number of digits within $s_n$ to the number $v$.
Let us see why  equality  \eqref{eq:4} 
     holds.
We prove that for 
 for sufficiently large $N$
     \[
     10^n/N < 2 / \log N,
     \]
Using $L\leq N$ we have,
\setlength{\mathindent}{0pt}
     \[
     \frac{2N}{10^n} \geq \frac{2}{10^n}\left( (n-1)10^{n-1} - \frac{10^{n-1}}{ 9} + \frac{1}{9}\right) \geq 2\left((n-1) - \frac{1}{9}\right)\frac{1}{10} \geq 2\left (n - 1 - \frac{1}{9}\right) \geq n + \log n.
     \]
     And using  $N\leq L+M$ we have
\setlength{\mathindent}{35pt}
\[
     \log N < n + \log n,
\]
So,
     \[
     \log N < n + \log n < 2 N/10^n.
     \]
We obtain,
     $|D(c,\alpha, \alpha + 10^{-k}, N)| = O\left(\frac{10^{-k}}{\log N}\right)$. 

     \textbf{Step 2:} Let $h \in \N$ and $\gamma \in [0,1)$, $\gamma = 0.\gamma_1 \ldots \gamma_h = \sum _{i=1}^h \gamma_i 10^{-i}$. 
     We show
     \[
     |D(c,0, \gamma, N)| = O \left(\frac{1}{\log N}\right),
     \]
     whose constant does not depend on $\gamma$. 
This is the desired result just for intervals whose extremes $\gamma$ are numbers with finite decimal expansion.
     Let $k \in \{1,\ldots,h\}$ and $j \in \{0,\ldots,\gamma_k\}$. 
     \\
     We define 
     \[
     \lambda _{k,j} = \sum\limits _{i=1}^{k-1} \gamma_i 10^{-i} + j 10 ^{-k} = 0.\gamma_1 \ldots \gamma_{k-1} \: j.
     \]
     Then, it holds that 
     \[
     \lambda _{k, j+1} = \sum\limits _{i=1}^{k-1} \gamma_i 10^{-i} + (j+1)10^{-k} = \lambda_{k,j} + 10^{-k}.
     \]
     We observe
     \[
     \begin{array}{llll}
          \lambda _{1,0} = 0 \:, & \lambda _{1,1} = 0.1 \:, & \ldots & \lambda _{1, \gamma_1} = 0.\gamma_1 \\
          \lambda _{2,0} = 0.\gamma_1 \:, & \lambda _{2,1} = 0.\gamma_1 1 \:, & \ldots & \lambda _{2, \gamma_2} = 0.\gamma_1 \ gamma_2 \\
          \vdots & \vdots & & \vdots
          \\
          \lambda _{h,0} = 0.\gamma_1\ldots\gamma _{h-1} \:, & \lambda _{h,1} = \gamma_1\ldots\gamma _{h-1}1 \:, & \ldots & \lambda _{h, \gamma_h} = 0.\gamma_1 \ldots \gamma_h = \gamma.
     \end{array}
     \]
     Then,
    \begin{align*}
           \sum_{k=1}^{h} \sum_{j=0}^{\gamma_k - 1} D(c,\lambda_{k,j}, \lambda_{k, j+1}, N) &= \frac{1}{N} \sum_{n=1}^{N} \sum_{k=1}^{h} \sum_{j=0}^{\gamma_k - 1} \chi_{[\lambda_{k,j}, \lambda_{k, j+1})} (x_n) - \sum_{k=1}^{h} \sum_{j=0}^{\gamma_k - 1} 10^{-k} \\
           &= \frac{1}{N} \sum_{n=1}^{N} \chi_{[0,\gamma)} (x_n) - \sum_{k=1}^{h} \gamma_k 10^{-k} \\
           &= \frac{1}{N} \sum_{n=1}^{N} \chi_{[0,\gamma)} (x_n) - \gamma \\
           &= D(c,0, \gamma, N).
        \end{align*}
     Hence,
     \begin{align*}
             |D(c,0, \gamma, N)| &\leq \sum _{k=1}^{h} \sum _{j=0}^{\gamma_k - 1} |D(c,\lambda _{k,j}, \lambda _{k, j+1}, N)| \\
             &= \sum _{k=1}^{h} \sum _{j=0}^{\gamma_k - 1} D(c,\lambda _{k,j}, \lambda _{k, j} + 10^{-k}, N) \\
             &= \sum _{k=1}^{h} O \left( \frac{10^{-k}}{\log N} \right) \qquad \text{by step 1, taking $N$ sufficiently large}\\
             &= O \left( \frac{1}{\log N} \right)
         \end{align*}
where the hidden constant in $O \left(\frac{1}{\log N} \right)$ does not depend on $\gamma$, nor on $h$, since
     \[
     \sum_{k=1}^{h} O \left( \frac{10^{-k}}{\log N} \right) \leq O \left(\frac{1}{\log N}\ \right) \sum _{k=1}^\infty 10^{-k} \leq O\left(\frac{1}{\log N}\right).
     \]

     \textbf{Step 3:} Let $\beta \in [0, 1)$.
     We prove that $|D(c,0, \beta, N)| = O \left( \frac{1}{\log N} \right)$. 
     Let $N \in \N$. Let's take $h = [\log(\log N))]$.
     If $\beta$ has a finite decimal expansion, then we are in the case of Step 2 and the proof is complete. Otherwise, $\beta = 0.\beta_1 \beta_2 \ldots \beta_h \beta _{h+1}\ldots$. 
     Let $\alpha, \gamma \in [0,1)$ be such that
     \begin{align*}
     &\alpha \leq \beta \leq \gamma, \\
     &\gamma - \alpha = 10^{-h}, \\
     &\alpha 10^h \in \N, \\
     &\gamma 10^h \in \N.
     \end{align*}
     That is to say,
     $
     \alpha = 0.\beta_1 \ldots \beta_h \:\text{ and } \: \gamma = \alpha + 10^{-h}.
     $
     So,
     \begin{align*}
     D(c,0, \beta, N) &= \frac{1}{N} \sum _{n=1}^N \chi _{[0,\beta)} (x_n) - \beta 
     \\&\leq \frac{1}{N } \sum_{n=1}^N \chi_{[0,\beta)} (x_n) - \gamma + \gamma - \alpha = D(c,0, \gamma, N) + 10^{-h}.
     \end{align*}
     Similarly,
    \[
    D(c,0, \beta, N) \geq \frac{1}{N} \sum_{n=1}^N \chi_{[0,\beta)} (x_n) - \alpha + \alpha - \gamma = D(c,0, \alpha, N) - 10^{-h}.
    \]
     Then, $D(c,0, \alpha, N) - 10^{-h} \leq D(c,0, \beta, N) \leq D(c,0, \gamma, N) + 10^{-h}$. 
     Therefore,
     \begin{align*}
             |D(c,0, \beta, N)| &\leq \max \{ |D(c,0, \alpha, N) - 10^{-h}|, |D(c,0, \gamma, N) + 10^{-h}| \} \\
             &\leq \max \{ |D(c,0, \alpha, N)|, |D(c,0, \gamma, N)| \} + 10^{-h} \\
             &= O \left( \frac{1}{\log N} \right) + 10^{-h} \qquad \text{by step 2}\\
             &= O \left( \frac{1}{\log N} \right).
     \end{align*}
     The proof of Theorem \ref{thm:1} is complete.\qed

\section{Theorem~\ref{thm:2}: Lower Bound}

We prove a  lower bound for the discrepancy $D(c,N)$. 
We need to find out how much  the number of occurrences of  the blocks of equal length can differ.
It suffices to find only two \textit{witnessing blocks} 
 one whose occurrences are in excess and the other in defect.

\subsection{Witnessing Blocks}

The statement of Lemma~\ref{lem:2}  is due to Schiffer~\cite[Lemma 2]{Schiffer}. 
This is our version of the proof.

\begin{lemma} \label{lem:2}
     Let $\alpha \in [0,1)$.
 Let $B_1$ and $B_2$ be blocks of equal length $k$.
Suppose there  exists a constant $C>0$ such that for infinitelly many $N \in \N$,
     \[
     \left|\Nu(\alpha, B_1, N) - \Nu(\alpha, B_2, N) \right| > C \frac{N}{\log(N)}.
     \]
 Then, there exists a constant $K > 0$ such that for infinitely many  $N \in \N$,
     \[
     D(N, \alpha) > \frac{K}{\log(N)}.
     \]
     Furthermore, it holds for the constant $K = C/3$. 
\end{lemma}

\begin{proof}
     Let $B_1 = (b_1 \ldots b_k)$, $B_2 = (d_1 \ldots d_k)$ and $\alpha = 0.\alpha_1 \: \alpha_2 \: \alpha_3 \ldots$ $\in [0,1) $. Let 
     \[
     \begin{array}{lll}
         \beta_1 = 0.b_1 \ldots b_k = \sum _{i=1}^k b_i 10^{-i} & \text{ and } & \beta_2 = 0.d_1 \ldots d_k = \sum _{i= 1}^k d_i 10^{-i}
         \\
         I_1 = [\beta_1, \: \beta_1 + 10^{-k}) \subseteq [0,1) & \text{ and } & I_2 = [\beta_2, \: \beta_2 + 10^{-k }) \subseteq [0,1).
     \end{array}
     \]
 Observe that $\alpha \in I_1 \iff (\alpha_1 \ldots \alpha_k) = (b_1, \ldots, b_k)$. 
     And $\alpha \in I_2 \iff (\alpha_1 \ldots \alpha_k) = (d_1, \ldots, d_k)$. 
     Let $N \in \N$ be such that $N + k - 1$ satisfies the hypothesis of the statement, 
     \[
     \left|\Nu(\alpha, B_1, N+k-1) - \Nu(\alpha, B_2, N+k-1) \right| > C \frac{N+k-1}{\log(N+k-1)}.
     \]
     So,   
    \begin{align}
        D(\alpha,N) 
        & = D( (10^{n-1} \alpha \mod 1,N)_{n \in \N}) \notag
        \\
        & = \sup_{0\leq a<b<1} 
        \left| \frac{\# \{ n \in \{1, \ldots, N\}: 10^{n-1} \alpha \mod 1 \in[a,b)\}}{N} - (b-a) \right| \notag
        \\
        & \geq \max_{i=1,2}
        \left| \frac{\# \{ n \in \{1, \ldots, N\}: 10^{n-1} \alpha \mod 1 \in I_i\}}{N} - 10^{-k} \right| \notag
        \\
        & = \max_{i=1,2}
        \left| \frac{\Nu(\alpha, B_i, N+k-1)}{N} - 10^{-k} \right| 
        \label{eq:5} 
        \\
        & \geq \frac{ \left| \Nu(\alpha, B_1, N+k-1) - \Nu(\alpha, B_2, N+k-1) \right|}{2N} 
        \label{eq:6} 
        \\
        & > \frac{C(N+k-1)}{2N \log(N+k-1)} \notag
        \\
        & \geq \frac{C}{2 \log(N+k-1)}. \notag
    \end{align}
    
    To see  inequality \eqref{eq:5} notice that
$     10^{n-1} \alpha \mod 1 \in I_1,
\iff
     \alpha_n = b_1, \ldots, \alpha _{n+k-1} = b_k $,
and this happens exactly when 
$     B_1 \text{ occurs in } (\alpha_n , \ldots, \alpha _{n+k-1}).$
The same holds  for $B_2$. 
Observe that when $n=1$,
 we consider the occurrences in $(\alpha_1 \ldots \alpha_k)$, 
and when $n = N$, 
we  consider the occurrences in $(\alpha_N, \ldots, \alpha _{N+k-1})$. 

     To see inequality \eqref{eq:6} notice that\\
     \medskip
     
     $         \frac{ \left| \Nu(\alpha, B_1, N+k-1) - \Nu(\alpha, B_2, N+k-1) \right|}{2N}=
$
\begin{align*}
         &= \frac{1}{2} \left| \frac{\Nu(\alpha, B_1, N+k-1)}{N} - 10^{-k} \right. 
 \left. {} - \left( \frac{\Nu(\alpha, B_2, N+k-1)}{N} - 10^{-k} \right) \right|\\
         &\leq \frac{1}{2}
         \left( \left| \frac{\Nu(\alpha, B_1, N+k-1)}{N} - 10^{-k}\right| \right.
         \qquad \left. {} + \left| \frac{\Nu(\alpha, B_2, N+k-1)}{N} - 10^{-k}\right| \right) \\
         &\leq \frac{1}{2} 2 \max _{i=1,2} \left| \frac{\Nu(\alpha, B_i, N+k-1)}{N} - 10^{-k} \right|.
     \end{align*}
     
We obtained that $D(\alpha,N)\geq C/(1\log(N+k-1))$.
To finish the proof we can take $K=C/3$ because, since  $k$ is fixed, for sufficiently large $N$, 
         \[
         \frac{C}{2\log(N+k-1)} \geq \frac{C}{3\log(N)}.
         \]
\end{proof}

\subsection{Proof of the Theorem~\ref{thm:2}}

We use Lemma~\ref{lem:2}. Let $B_1$ and $B_2$ be of equal length $k \geq 2$, 
 $B_2 = (0 \ldots 0)$ the block of all zeros, and $B_1 = (1 \: 1 *\ldots *)$, 
where the asterisk~$*$ represents any digit between $0$ and $9$.
\medskip


     \begin{example*}
     Let $k=2$, $B_1 = (1\:1)$ and $B_2 = (0\:0)$, 
Let's see that $B_1$ has occurrences in excess and $B_2$ in defect. 
Notice that $B_1$ has overlapping occurrences in the expansion of $c$ 
bit  $B_2$ does not, because $B_2$ is the block of all zeros and there aren't terms $t_i$ 
beginning with a $0$. 
For the non-overlapping occurrences of each block  we must count.
Let's count of the non-overlapping occurrences of $B_1$ and $B_2$ in 
$(1 \ldots 999) = (s_1 \: s_2 \: s_3)$; that is, we want to calculate for $i = $1.2,
     \[
     \sum_{\ell = 1}^{3} \U_{no} (B_i, s_\ell).
     \]
 We count up to $\ell = 3$, but the procedure is similar for  every $\ell$. 
     
     {\em Case of $B_1$:}
     \begin{itemize}
         \item $\U_{no} (B_1, s_1) = 0$, because $k =2> 1$.
         \item $\U_{no} (B_1, s_2) = 1$, because the only occurrence of $B_1$ in $s_2$ is $(1 \: 1)$.
         \item $\U_{no} (B_1, s_3) = 19$, because there are ten occurrences of the form $(1 \: 1 \: *)$ and nine of the form $(* \: 1 \: 1)$.
     \end{itemize}
     Thus,
   $
     \sum_{\ell = 1}^{3} \U_{no} (B_1, s_\ell) = 20.
     $

     {\em Case of $B_2$:}
     \begin{itemize}
         \item $\U_{no} (B_1, s_1) = 0$, because $k > 1$.
         \item $\U_{no} (B_1, s_2) = 0$, because the term $(0\:0)$ does not appear in the $c$.
         \item $\U_{no} (B_1, s_3) = 9$ because there are nine occurrences of the form $(* \: 0 \: 0)$, but none of the form $(0 \: 0 \: *)$ .
     \end{itemize}
     Thus,
$     \sum_{\ell = 1}^{3} \U_{no} (B_1, s_\ell) = 9.
 $

\flushright{\qedhere}
\end{example*}

     In order to use Lemma~\ref{lem:2}, we must see that there exists a constant $C$ such that for infinitely many $N \in \N$,
     \[
     \left|\Nu(c, B_1, N) - \Nu(c, B_2, N) \right| > C \frac{N}{\log(N)},
     \]
     where $B_1$ and $B_2$ are the witnessing blocks. 
     Let $N \in \N$. Observe that
     $     \left|\Nu(c, B_1, N) - \Nu(c, B_2, N) \right| $ is equal to 
     \[
 \left|\Nu_{no}(c, B_1, N) + \Nu_o(c, B_1, N) - \left(\Nu_{no}(c, B_2, N) + \Nu_o(c, B_2, N) \right) \right|.
     \]
    We first calculate $\Nu_{no}(c, B_1, N) - \Nu_{no}(c, B_2, N)$.
     For $\ell \in \N$, we calculate $\U_{no}(B_1, s_\ell) - \U_{no}(B_2, s_\ell)$. Recall in Lemma \ref{lem:1} we define $\U_{no}(B_i, s_\ell)$ as the number of non-overlapping occurrences of $B_i$ in 
     \[
     s_{\ell} = (10^{\ell -1}, \ldots , 10^{\ell}-1) = (\underbrace{10\ldots0}_{\ell \text{ digits}}, \: \ldots \: , \underbrace{9\ldots9}_{\ell \text { digits}}).
    \]
     
     If $\ell<k$: $\U_{no}(B_i, s_\ell) = 0$, for $i = 1,2$.
     
     If $\ell \geq k$:

     $    \U_{no}(B_1, s_\ell) = 10^{\ell-k} + (\ell - k) \cdot 9 \cdot 10^{\ell-k-1}$, because the first digit of $B_1$ is not $0$.
         \\
   \hspace*{1.8cm}     $ \U_{no}(B_2, s_\ell) = (\ell - k) \cdot 9 \cdot 10^{\ell-k-1}$, because the first digit of $B_2$ is $0$.
   
     Therefore,     
     \begin{align} 
     \label{eq:7}
         \U_{no}(B_1, s_\ell) - \U_{no}(B_2, s_\ell) = 10^{\ell-k} ,\: \text{ for all }\ell \geq k.
     \end{align}

     \begin{remark*}
         We are using that the first digit of $B_2$ is zero and the first digit of $B_1$ is not zero. 
     \end{remark*}

Let $v=v(N)$ and $n=n(N)$. As in the proof of Lemma \ref{lem:1}, for $i=1,2$ it holds that
     \[
     \Nu_{no} (B_i, c, N) = \sum_{\ell = k}^{n-1} \U_{no}(B_i, s_\ell) + \Nu_{no}(B_i, v) -O(n)
     \]
We must calculate $\Nu_{no}(B_1, v) - \Nu_{no}(B_2, v)$.
     
     If $n < k$: $\Nu_{no}(B_i, v) = 0$. 
     
     If $n \geq k$: As in the proof of Lemma \ref{lem:1} 
let $v = v_1 \dots v_n$, and for each $j\in\{0,\dots,n-k\}$,
let  $a_j = v_1 \dots v_j$. Let  $\Nu_{no}(B, v, j)$ be the amount of numbers of the form
     \[
     y_j = \underbrace{*\ldots*}_{j \text{ digits}} \: \: B \underbrace{*\ldots*}_{n - k - j \text{ digits}} \: \text{ with } y_j \leq v.
     \]
 
{\em Case ${j=0}$}. We  count the amount of numbers of the form
     \[
     y_0 = B_i \underbrace{*\dots*}_{n-k \text{ digits}} \quad \text{con } y_0 \leq v.
     \]
Since the first digit of $B_2$ is zero, 
\[
\Nu_{no}(B_2, v, 0) = 0.
\]

Again by proof of Lemma \ref{lem:1}, 
\[
     \Nu_{no}(B_1, v, 0) \leq 10^{n-k} .
 \]

 {\em Case  ${1 \leq j \leq n-k}$}.
The amount of numbers of the form $y_j$ 
with the first $j$ digits less than $a_j$ , then  $y_j < v$,
and then $\Nu_{no}(B_1, v, j)=\Nu_{no}(B_2, v, j)$.
If the first $j$ digits of $y_j$ are greater than $a_j$, then $y_j > v$ 
so  they do not add up to $ \Nu_{no}(B_1, v, j)$ nor  to $\Nu_{no}(B_2, v, j)$. 
     Finally,  if  the first $j$ digits of $y_j$ are equal to $a_j$. 
There, the number of non-overlapping occurrences  of $B_1$ could be different 
from that of $B_2$. We analyze this case.
     Let  $\Delta_j(B_i)$ be  the amount of numbers of the form
     \[
     y_j = v_1 \ldots v_j \: B_i \!\!\! \underbrace{* \ldots *}_{n-k-j \text{ digits}}, \quad \text{ with } y_j \leq v.
     \]
     \begin{itemize}
         \item If $B_i > v_{j+1} \dots v_{j+k} : \: \Delta_j(B_i) = 0$, then in that case $y_j > v$.
         \item If $B_i = v_{j+1} \dots v_{j+k}: \: \Delta_j(B_i) = v_{j + k +1} \dots v_n + 1$, then the last $n-k-j$ they can take any value from
 $0$ to $v_{j + k+1} \ldots v_n$.
         \item If $B_i < v_{j+1} \dots v_{j+k} : \: \Delta_j(B_i) = 10^{n-k-j}$, then 
the last $n-k-j$ can take any value from 0 to $ \underbrace{9\dots9}_{n-j-k}$.
     \end{itemize}
     Therefore, $\Delta_j(B_i) \leq 10^{n-k-j},  \text{ for each } j \in \{1,\dots,n-k\}$. 
So, for $N$ large enough so that $n \geq k$, 
     \begin{align} 
     \label{eq:8}
         \Nu_{no}(B_1, v) - \Nu_{no}(B_2, v) &= \Nu_{no}(B_1, v, 0) + \sum_{j=1}^{n-k} (\Nu_{no}(B_1, v, j) - \Nu_{no}(B_2, v, j)) \\
        & = \Nu_{no}(B_1, v, 0) + \sum _{j=1}^{n-k} (\Delta_j(B_1 ) - \Delta_j(B_2)).\nonumber
     \end{align}

     \begin{remark*}
         Since $B_2 \leq B_1$, then $0 \leq \Delta_j(B_1) \leq \Delta_j(B_2)$. Hence,
         \[ v_1 \ldots v_j \: B_1 \underbrace{*\dots*}_{\substack{n-k-j\\ \text{digits}}} \: \leq \: v \quad \text{ implies }
         \quad v_1 \ldots v_j \ : B_2 \underbrace{*\dots*}_{\substack{n-k-j\\ \text{ digits}}} \: \leq \: v.
         \]
     \end{remark*}
Hence,
     \begin{equation}
          \label{eq:9}
         0 \leq \Delta_j(B_1) - \Delta_j(B_2) \leq 10^{n-k-j}.
     \end{equation}
We conclude that  for all $N$ large enough so that $n \geq k$,

     \setlength{\mathindent}{35pt}

     \begin{align} 
     &\Nu_{no}(c, B_1, N) - \Nu_{no}(c, B_2, N)=
 \label{eq:10}
             \\
             &= \sum_{\ell=k}^{n-1} \U_{no}(B_1, s_\ell) + \Nu_{no}(B_1, v) - O(n) - \left( \sum_{\ ell=k}^{n-1} \U_{no}(B_2, s_\ell) + \Nu_{no}(B_2, v) - O(n) \right)
   \nonumber 
             \\
             &= \sum_{\ell=k}^{n-1} \U_{no}(B_1, s_\ell) - \U_{no}(B_2, s_\ell) + \Nu_{no}(B_1, v) - \ \Nu_{no}(B_2, v) + O(n)
             \nonumber
             \\
             &= \sum_{\ell=k}^{n-1} 10^{\ell-k} + \Nu_{no}(B_1, v) - \Nu_{no}(B_2, v) + O(n ) \qquad \text{using } (\ref{eq:7})
             \nonumber
             \\
             &= \sum _{\ell=k}^{n-1} 10^{\ell-k} + \Nu_{no}(B, v, 0) + \sum _{j=1}^{n-k} \Delta_j(B_1) - \Delta_j(B_2) + O(n) \qquad \text{using } (\ref{eq:8})
             \nonumber\\
             &\geq \sum _{\ell=k}^{n-1} 10^{\ell-k} + \sum _{j=1}^{n-k} \Delta_j(B_1) - \Delta_j(B_2) + O (n)
             \nonumber
             \\
             &=\sum _{\ell=k}^{n-1} 10^{\ell-k} - \sum _{j=1}^{n-k} \Delta_j(B_2) - \Delta_j(B_1) + O( n)
             \nonumber
             \\
             &\geq \sum _{\ell=k}^{n-1} 10^{\ell-k} - \sum _{j=1}^{n-k} 10^{n-j-k} + O(n) \qquad \text{using } (\ref{eq:9}).
             \nonumber
\end{align}
     \setlength{\mathindent}{35pt}

     We now count the overlapping occurrences. 

     \begin{remark*}
         Since $B_2$ is the block of all zeros, it has no overlapping occurrences (since no number starts with leading zeros). 
This explains the choice of  $B_2$.
     \end{remark*}

We count  the overlapping occurrences of $B_1$. 
It is enough for us to count some of them,
enough so as to  verify the hypothesis of Lemma~\ref{lem:2}. 

     Recall that $B_1 = (1 \: 1\underbrace{*\dots*}_{k-2 \text{ digits}}) = (b_1 b_2 \dots b_k)$. 
Then, we define for each $\ell \in \N$,
     \[
     P_{\ell} = \{m \in \N : m = \underbrace{b_2 b_3 \dots b_k \overbrace{* \dots *}^{\ell-k \text{ digits}} b_1}_{\ ell \text{ digits}} = b_2 10^{\ell-1} + b_3 10^{\ell-2} + \ldots + b_k 10^{\ell - k +1} + \ldots + b_1 \}
     \]
     That is, $P_\ell$ is the set numbers  of $\ell$ digits, 
whose first $k-1$ digits are $(b_2 \: b_3 \dots b_k)$, and its last digit is $b_1$ .
     
     If $\ell<k$: $\#P_{\ell} = 0$. 
     
     If $\ell \geq k$: $\#P_{\ell} = 10^{\ell-k}$, then there are $\ell - k$ free digits. 

     \begin{remark*}
         We are using that $b_2 \neq 0$, otherwise we would be looking for $m \in \N$ that starts with a leading zero.
     \end{remark*}
     Observe  that for every element of $P_{\ell}$, there is an occurrence overlapping $B_1$ between two $\ell$ digit numbers. Then,
     \begin{equation}
     \label{eq:11}
         \U_o(B_1, s_\ell) \geq 10^{\ell-k}, \:\:\: \forall \ell \geq k
     \end{equation}
     where, remember, $\U_o(B_1, s_\ell)$ was the number of overlapping occurrences $B_1$ in $s_\ell$.

     \begin{remark*}
         We are using:
         \begin{itemize}
             \item $k \neq 1$, otherwise there would be no occurrences around $B_1$. 
             
            \item $b_1 \neq 9$, because if not, it could happen that $m = b_2 \: 9 \dots 9$ and then $m+1$ does not have $b_2$ as its first digit, and therefore, it does not produce an overlapping.
         \end{itemize}
     \end{remark*}
    
Therefore, for all $N$ large enough so that  $n \geq k$,
\medskip

     \setlength{\mathindent}{0pt}
$
    \Nu(c, B_1, N) - \Nu(c, B_2, N) =
$
 \setlength{\mathindent}{35pt}
     \begin{align}
         &= \Nu_{no}(c, B_1, N) + \Nu_o(c, B_1, N) ) - \left(\Nu_{no}(c, B_2, N) + \Nu_o(c, B_2, N) \right)
             \nonumber
             \\
             &= \Nu_{no}(c, B_1, N) - \Nu_{no}(c, B_2, N) + \Nu_o(c, B_1, N) - \underbracket{\Nu_o( B_2, c, 1, N)}_{\substack{= 0}}
             \nonumber
             \\
             &\geq \sum _{\ell=k}^{n-1} 10^{\ell-k} - \sum _{j=1}^{n-k} 10^{n-j-k} + O(n) + \Nu_o(c,B_1,  1, N) \qquad \text{using } (\ref{eq:10})
             \nonumber
             \\
             &\geq \sum_{\ell=k}^{n-1} 10^{\ell-k} - \sum _{j=1}^{n-k} 10^{n-j-k} + O(n) + \sum_ {l=1}^{n-1} \U_o(B_1,s_\ell)
             \nonumber
             \\
             &\geq \sum_{\ell=k}^{n-1} 10^{\ell-k} - \sum _{j=1}^{n-k} 10^{n-j-k} + O(n) + \sum_ {\ell=k}^{n-1} 10^{\ell-k} \qquad \text{using } (\ref{eq:11})
             \nonumber
             \\
             &= 2\sum _{\ell=k}^{n-1} 10^{\ell-k} - \sum _{m=k}^{n-1} 10^{m-k} + O(n) \qquad \text{change of variable: } m=n-j
             \nonumber
             \\
             &= \sum_{\ell=k}^{n-1} 10^{\ell-k} + O(n)
             \nonumber
             \\
             &\geq \frac{1}{10^{k+1}} 10^n
             \nonumber
             \\
             &\geq \frac{1}{10^{k+1}}  \frac{N}{\log(N)}. \label{eq:final}
     \end{align}  
To see the last inequality observe that, since $n$ is the number of digits of $v$, 
     \[
     \sum _{j=1}^{n-1} \underbracket{j}_{\substack{\text{Number of}\\ \text{digits of}\\ \text{each number} }} \cdot \: \underbracket{9 \cdot 10^{j-1}}_{\substack{\text{Quantity}\\ \text{of numbers}\\ \text{in } s_j}} \leq N \leq \sum_ {j=1}^{n} j \cdot 9 \cdot 10^{j-1}.
     \]
     Furthermore, for all $r \in \N$, $\sum\limits_{j=1}^{r} j \cdot 9 \cdot 10^{j-1} = 10^r r - \frac{10^ r}{9} + \frac{1}{9}$.
     Then,
   \[
     N \geq \sum _{j=1}^{n-1} j \cdot 9 \cdot 10^{j-1} = 10^{n-1} n-1 - \frac{10^{n-1 }}{9} + \frac{1}{9} > 10^{n-1} \left(n-1-\frac{1}{9}\right)
     \]
     Therefore,
\[\log(N) \geq n-1 + \log\left(n - \frac{10}{9}\right).
     \]
And
 \[
     \log(N)10^n 
     \geq  \left( n-1 + \log(n - \frac{10}{9}) \right) 10^n 
     \geq \left( 10^n n - \frac{10^n}{9} + \frac{1}{9} \right) 
     =  \sum _{j=1}^{n} j \cdot 9 \cdot 10^{j-1} 
     \geq  N.
     \]

     Finally, we  make the  constant $K$ explicit.    
 We  just proved in  (\ref{eq:final}) that  
 \[ \Nu(c, B_1, N) - \Nu(c, B_2, N) \geq C \frac{N}{\log(N)}
 \] with $C=\frac{1}{10^{k+1}} $.
Since since the smallest possible value of  $k$ that we can take is $k=2$, and by   Lemma~\ref{lem:2}  we can take $K = C/3$,   we obtain
     \[
     K = \frac{C}{3} = \frac{1}{10^{k+1}  3} = \frac{1}{10^3  3}.
     \]

    The statement of Theorem~\ref{thm:2} asks the bound  for infinitely many $N$. 
Since we gave the lower bound   for 
all $N$ sufficiently large so that $n(N) \geq k$,
we gave it for   for cofinitely many $N$, a stronger result.
The proof of  Theorem \ref{thm:2} is complete. $\qed$


\section*{Acknowledgements}
This work was supported by grant  UBACYT 20020220100065BA from Universidad de Buenos Aires.
\bibliographystyle{plain}

\bibliography{champ}
\bigskip

\begin{minipage}{\textwidth}
\small

\noindent
Verónica Becher
\\Departamento de Computación,
Facultad de Ciencias Exactas y Naturales e ICC\\
Universidad de Buenos Aires y CONICET\\
Pabellón 0, Ciudad Universitaria
(C1428EGA) Buenos Aires, Argentina\\
{\tt vbecher@dc.uba.ar}
\bigskip

\noindent
Nicole Graus
\\Departamento de Matemática,
Facultad de Ciencias Exactas y Naturales
Universidad de Buenos Aires
Pabellón 1, Ciudad Universitaria
(C1428EGA) Buenos Aires, Argentina
\\
{\tt nicole.graus.h@gmail.com}
\end{minipage}
\end{document}